\documentclass[12pt,twoside]{article}
\usepackage[cp1251]{inputenc}
\usepackage{amsmath,amsopn,amsthm,amscd,amsfonts,longtable,amssymb,srcltx,comment}
\usepackage[final]{graphicx}
\usepackage{epsfig}
\usepackage{euscript}
\usepackage{srcltx,epsf}
\usepackage{wrapfig}

\newtheorem{lemma}{Lemma}
\newtheorem{theorem}{Theorem}

\newtheorem{corollary}{Corollary}
\newtheorem{proposition}{Proposition}

\theoremstyle{definition}

\DeclareMathOperator{\ch}{ch}
\DeclareMathOperator{\sh}{sh}

\DeclareMathOperator{\csch}{csch}

\begin{document}

\markboth{N.~Abrosimov \& B.~Vuong}
{Explicit volume formula for a hyperbolic tetrahedron in terms of edge lengths}


\title{Explicit volume formula\\ for a hyperbolic tetrahedron\\ in terms of edge lengths}

\author{Nikolay Abrosimov\footnote{Corresponding author.}}


\author{Nikolay Abrosimov, Bao Vuong}


\maketitle

\begin{abstract}
We consider a compact hyperbolic tetrahedron of a general type. It is a convex hull of four points called vertices in the hyperbolic space $\mathbb{H}^3$. It can be determined by the set of six edge lengths up to isometry. For further considerations, we use the notion of edge matrix of the tetrahedron formed by hyperbolic cosines of its edge lengths. 

We establish necessary and sufficient conditions for the existence of a tetrahedron in $\mathbb{H}^3$. Then we find relations between their dihedral angles and edge lengths in the form of a cosine rule. Finally, we obtain exact integral formula expressing the volume of a hyperbolic tetrahedron in terms of the edge lengths. The latter volume formula can be regarded as a new version of classical Sforza's formula for the volume of a tetrahedron but in terms of the edge matrix instead of the Gram matrix.
\end{abstract}



\section{Introduction}
A {\em hyperbolic tetrahedron} $T$ is a convex hull of four points in the hyperbolic space $\mathbb{H}^3$. These points are called {\em vertices} of $T$. Let us denote them by numbers $1, 2, 3$ and $4$ (see Fig.~1). Then denote by $\ell_{ij}$ the length of the edge connecting $i$-th and $j$-th vertices. We put $\theta_{ij}$ for the dihedral angle along the corresponding edge.
\begin{figure}[th]
	\centerline{\psfig{file=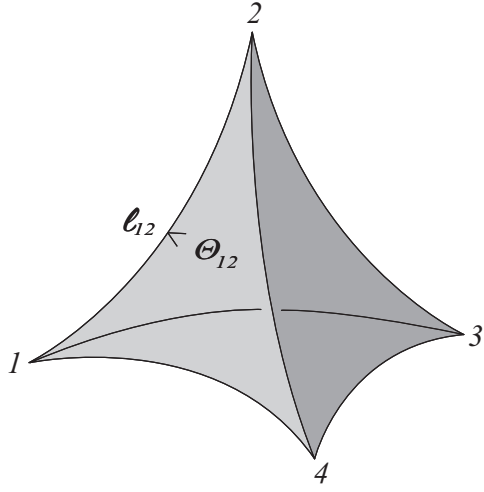,width=2.5in}}
	\caption{Hyperbolic tetrahedron $T$.\label{fig1}}
\end{figure}

A {\em Gram matrix} $G(T)$ of tetrahedron $T$ is defined as $G(T)=$
$$
\langle-\cos\theta_{ij}\rangle_{i,j=1,2,3,4}=
\left(%
\begin{array}{cccc}
1 & -\cos\theta_{12} & -\cos\theta_{13} & -\cos\theta_{14} \\
-\cos\theta_{12} & 1 & -\cos\theta_{23} & -\cos\theta_{24} \\
-\cos\theta_{13} & -\cos\theta_{23} & 1 & -\cos\theta_{34} \\
-\cos\theta_{14} & -\cos\theta_{24} & -\cos\theta_{34} & 1 \\
\end{array}%
\right),
$$ 
we assume here that $-\cos\theta_{ii}=1$.

An {\em Edge matrix} $E(T)$ is formed by hyperbolic cosines of the edge lengths and defined as follows
$$E(T)=\langle\ch\ell_{ij}\rangle_{i,j=1,2,3,4}=
\left(%
\begin{array}{cccc}
1 & \ch\ell_{12} & \ch\ell_{13} & \ch\ell_{14} \\
\ch\ell_{12} & 1 & \ch\ell_{23} & \ch\ell_{24} \\
\ch\ell_{13} & \ch\ell_{23} & 1 & \ch\ell_{34} \\
\ch\ell_{14} & \ch\ell_{24} & \ch\ell_{34} & 1 \\
\end{array}%
\right),
$$ 
where $\ell_{ii}=0$ and $\ch\ell_{ii}=1$.

It is known that a hyperbolic tetrahedron $T$ can be uniquely determined up to isometry either by the Gram matrix $G(T)$ or the edge matrix $E(T)$ (see, e.g., \cite{Vinberg}). This is unlikely to Euclidean case, where the edge matrix defines a tetrahedron up to isometry, but the Gram matrix defines a tetrahedron only up to similarity. The notion of similarity has no place in the hyperbolic geometry.

A volume formula for arbitrary hyperbolic tetrahedron has been unknown until recently. A general algorithm for obtaining such a formula was indicated by W.–-Yi.~Hsiang in \cite{Hsiang}. A complete solution of the problem was given by Yu.~Cho and H.~Kim \cite{Cho-Kim}. However, the proposed formula was asymmetric with respect to permutation of angles. J.~Murakami, M.~Yano \cite{Mur-Yan} obtained a formula expressing the volume by dihedral angles in a symmetric way. A.~Ushijima \cite{Ush} presented a simple proof of the Murakami–-Yano formula. He also investigated the case of a truncated hyperbolic tetrahedron. In all these studies the volume is expressed as a linear combination of 16 dilogarithms or Lobachevsky functions. The arguments of these functions depend on the dihedral angles of the tetrahedron and some additional parameter, which is the root of some quadratic equation with complex coefficients. In 2005, D.~A.~Derevnin and A.~D.~Mednykh \cite{Der-Med} presented an integral formula in terms of dihedral angles. 

Surprisingly, but more than 100 years before, in 1907, G.~Sforza \cite{Sforza} found another closed integral formula for the volume of a hyperbolic tetrahedron.

\begin{theorem}[G.~Sforza, 1907]
	Let $T$ be a compact hyperbolic tetrahedron given by the Gram matrix $G=G(T)$. We assume that all the dihedral angles are fixed exept $\theta_{34}$ which is formal variable. Then the volume $V=V(T)$ is given by the formula
	$$
	Vol\,(T)=\frac{1}{4}\int\limits_{t_0}^{\theta_{34}}\log\frac{c_{34}(t)-\sqrt{-\det\,G(t)}\sin t}{c_{34}(t)+\sqrt{-\det\,G(t)}\sin t}dt,
	$$
	where $t_0$ is a suitable root of the equation $\det G(t) = 0$, $c_{34}$ is
	$(3, 4)$-cofactor of the matrix $G$, and $c_{34}(t), G(t)$ are functions in one variable $\theta_{34}$ denoted by $t$.
\end{theorem}

In all the above mentioned formulas, the volume is given in terms of dihedral angles. 

In the paper by J.~Murakami and A.~Ushijima \cite{Mur-Ush} one can find a formula that express the volume of a hyperbolic tetrahedron in terms of edge lengths. However, the formula contains derivatives of implicit functions involving dilogarithms.

The natural question arises: {\it Can we find an analog of the Sforza's formula, but in terms of the edge matrix?} This would be the first known explicit formula for the volume of an arbitrary hyperbolic tetrahedron in terms of its edge lengths. 

In the present work we consider a general case of a compact tetrahedron $T$ in $\mathbb{H}^3$, given by its edge matrix. We establish necessary and sufficient conditions for the existence of a tetrahedron in $\mathbb{H}^3$ in terms of its edge lengths. Then we find relations between their dihedral angles and edge lengths in the form of a cosine rule. Finally, we obtain explicit integral formula for the volume of $T$ in terms of the edge matrix.

\section{Existence criterion of a tetrahedron $T$ in $\mathbb{H}^3$}
The following theorem gives a criterion for the existence of a compact hyperbolic tetrahedron in terms of its edge lengths.
\begin{theorem}\label{exist}
	A compact hyperbolic tetrahedron $T$ with edge matrix $E$ is exist if and only if the following inequalities hold \vspace{-10pt}
	\begin{eqnarray}\label{exi}
	\text{\upshape(i)} &  & \ell_{13}+\ell_{23}\geq \ell_{12}\geq|\ell_{13}-\ell_{23}|, \nonumber
	\\ \text{\upshape(ii)} &  & \ell_{14}+\ell_{24}\geq \ell_{12}\geq|\ell_{14}-\ell_{24}|, \nonumber
	\\ \text{\upshape(iii)} &  & \ell_{1}\leq \ell_{34}\leq \ell_{2},\quad where\quad\ch\ell_{1}=C-S, \quad\ch\ell_{2}=C+S\quad and \nonumber
	\\&  & C=\ch \ell_{13}\ch \ell_{14}-\csch^2\ell_{12}(\ch \ell_{13}\ch \ell_{12}-\ch \ell_{23})(\ch \ell_{14}\ch \ell_{12}-\ch \ell_{24}), \nonumber
	\\&  & S=\csch^2\ell_{12}\sqrt{(\ch\ell_{23}-\ch(\ell_{13}+\ell_{12}))(\ch\ell_{23}-\ch(\ell_{13}-\ell_{12}))} \nonumber
	\\&  & \quad\quad\times\sqrt{(\ch \ell_{24}-\ch(\ell_{14}+\ell_{12}))(\ch \ell_{24}-\ch(\ell_{14}-\ell_{12}))} \nonumber
	\end{eqnarray}
\end{theorem}	
\begin{proof}
	Without loss of generality we assume that length $\ell_{12}$ is equal to some finite positive number,  $0<\ell_{12}<\infty$. Then inequalities (i) and (ii) are necessary and sufficient for the existence of hyperbolic triangles with side lengths $\ell_{12}, \ell_{13}, \ell_{23}$ and $\ell_{12}, \ell_{14}, \ell_{24}$ correspondingly. 
	
	We construct a tetrahedron in $\mathbb{H}^3$ with two faces $1-2-3$ and $1-2-4$, adjacent along edge $1-2$. The two faces have edge lengths $\ell_{12}, \ell_{13}, \ell_{23}$ and $\ell_{12}, \ell_{14}, \ell_{24}$ (see Fig.~2). Such a tetrahedron is not unique since the distance between vertices $3$ and $4$ can be varied. By fixing the edge length $\ell_{34}$ we obtain a rigid tetrahedron.
	\begin{figure}[th]
		\centerline{\psfig{file=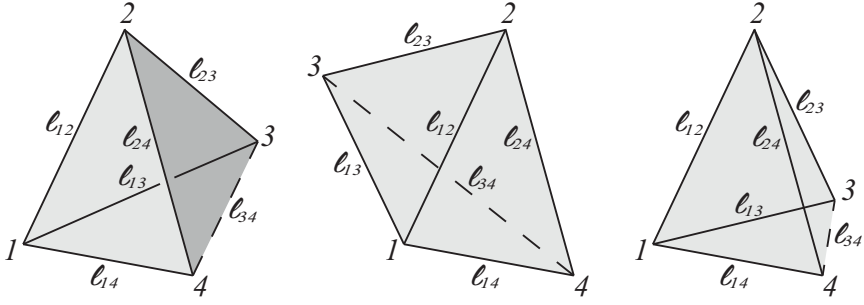,width=4.5in}}
		\caption{Tetrahedra with fixed edge lengths exept $\ell_{34}$.\label{fig1}}
	\end{figure}

	We show that  the edge length $\ell_{34}$ satisfies double inequality (iii) $\ell_1\leq\ell_{34}\leq\ell_2$. Consider a flexible construction consisting of two rigid triangles $1-2-3$ and $1-2-4$ with variable dihedral angle $\theta_{12}$ along the common edge $1-2$ (see Fig.~2). The construction degenerates to a flat hyperbolic quadrilateral as the dihedral angle $\theta_{12}$ is either $0$ or $\pi$.  Then the angle between edges $3-1$ and $1-4$ is as follows
	\begin{align*}
	\angle314&=|\,\angle214-\angle213\,|\quad{\rm if}\quad\theta_{12}=0,\\
	\angle314&=\angle214+\angle213\,\,\;\;\quad{\rm if}\quad\theta_{12}=\pi.
	\end{align*}
	Hence
	\begin{equation}\begin{split}\label{eq21}
	\cos\angle314&=\cos\angle214\,\cos\angle213+\sin\angle214\,\sin\angle213\quad{\rm if}\quad\theta_{12}=0,\\
	\cos\angle314&=\cos\angle214\,\cos\angle213-\sin\angle214\,\sin\angle213\quad{\rm if}\quad\theta_{12}=\pi.
	\end{split}\end{equation}
	By hyperbolic law of cosines for triangles $1-2-3$ and $1-2-4$ we have
	\begin{equation}\label{eq22}
	\cos\angle213=\frac{\ch\ell_{13}\ch\ell_{12}-\ch\ell_{23}}{\sh\ell_{13}\sh\ell_{12}}
	\quad{\rm and}\quad \cos\angle214=\frac{\ch\ell_{14}\ch\ell_{12}-\ch\ell_{24}}{\sh\ell_{14}\sh\ell_{12}}.
	\end{equation}
	Since $\sin x=\sqrt{1-\cos^2 x}$ for $0\leq x\leq\pi$, from (\ref{eq22}) we get
	\begin{equation}\begin{split}\label{eq23}
	\sin\angle213&=\sqrt{\frac{\sh^2\ell_{13}\sh^2\ell_{12}-\ch^2\ell_{13}\ch^2\ell_{12}-\ch^2\ell_{23}+2\ch\ell_{13}\ch\ell_{12}\ch\ell_{23}}{\sh^2\ell_{13}\sh^2\ell_{12}}}=\\
	&\frac{\sqrt{(\ch(\ell_{13}+\ell_{12})-\ch\ell_{23})(\ch\ell_{23}-\ch(\ell_{13}-\ell_{12}))}}{\sh\ell_{13}\sh\ell_{12}},\\
	\sin\angle314&=\sqrt{\frac{\sh^2\ell_{14}\sh^2\ell_{12}-\ch^2\ell_{14}\ch^2\ell_{12}-\ch^2\ell_{24}+2\ch\ell_{14}\ch\ell_{12}\ch\ell_{24}}{\sh^2\ell_{14}\sh^2\ell_{12}}}=\\
	&\frac{\sqrt{(\ch(\ell_{14}+\ell_{12})-\ch\ell_{24})(\ch\ell_{24}-\ch(\ell_{14}-\ell_{12}))}}{\sh\ell_{14}\sh\ell_{12}}.
	\end{split}\end{equation}
	By hyperbolic law of cosines for triangle $1-3-4$ we have
	\begin{equation}\label{eq24}
	\ch\ell_{34}=\ch\ell_{13}\ch\ell_{14}-\sh\ell_{13}\sh\ell_{14}\cos\angle314
	\end{equation}
	
	We set $\ell_1=\ell_{34}$ for $\theta_{12}=0$ and $\ell_2=\ell_{34}$ for $\theta_{12}=\pi$. Then the inequality $\ell_1\leq\ell_{34}\leq\ell_2$ holds. Substituting expressions (\ref{eq22}) and (\ref{eq23}) into equation (\ref{eq21}), we get $\cos\angle314$ in both cases $\theta_{12}=0, \pi$. Then we put it in (\ref{eq24}) to obtain $\ch\ell_1$ and $\ch\ell_2$, namely 
	\begin{align*}
	\ch\ell_{1}=\;&C-S,\quad\ch \ell_{2}=C+S,\quad {\rm where}\\
	C=\;&\ch \ell_{13}\ch \ell_{14}-\csch^2\ell_{12}(\ch \ell_{13}\ch \ell_{12}-\ch \ell_{23})(\ch \ell_{14}\ch \ell_{12}-\ch \ell_{24}),\\
	S=\;&\csch^2\ell_{12}\sqrt{(\ch \ell_{23}-\ch(\ell_{13}+\ell_{12}))(\ch \ell_{23}-\ch(\ell_{13}-\ell_{12}))}\\
	&\times\sqrt{(\ch \ell_{24}-\ch(\ell_{14}+\ell_{12}))(\ch \ell_{24}-\ch(\ell_{14}-\ell_{12}))}
	\end{align*}
	
	
	
	
\end{proof}

\section{Properties of the edge matrix of a tetrahedron $T$ in $\mathbb{H}^3$}
For further consideration we will use the following known property of a quadratic matrix (see, e.g., \cite{Jacobi}).
\begin{theorem}[Jacobi equation]\label{Jac}
	Let $M=(a_{ij})_{i,j=1,\ldots,n}$ be an $n\times n$ matrix. Denote by $C=(c_{ij})_{i,j=1,\ldots,n}$
	the matrix of cofactors $c_{ij}=(-1)^{i+j}M_{ij},$ where
	$M_{ij}$ is $ij$-th minor of matrix $M$. Then
	$$\det\,(c_{ij})_{i,\,j=1,\ldots,k}=\det M^{k-1}\cdot
	\det\,(a_{ij})_{i,\,j=k+1,\ldots,n}.$$
\end{theorem}

In the next theorem we show some properties of the edge matrix of a hyperbolic tetrahedron that we will need to derive a volume formula.
\begin{theorem}\label{edgematrix}
	Let $E$ be the edge matrix of a compact hyperbolic tetrahedron $T$. Then the following conditions hold 
	\begin{eqnarray}\label{exi}
	\text{\upshape(i)} &  & c_{ii}>0, \nonumber
	\\ \text{\upshape(ii)} &  & \det E<0, \nonumber
	\\ \text{\upshape(iii)} &  & \cos\theta_{5-i,5-j}=\frac{-c_{ij}}{\sqrt{c_{ii}\cdot c_{jj}}}, \nonumber
	\end{eqnarray}
	where $i,j\in\{1,2,3,4\},\;c_{ij}=(-1)^{i+j}E_{ij}$ is $ij$-cofactor of edge matrix $E$ and $\theta_{5-i,5-j}$ is a dihedral angle along edge $\ell_{5-i,5-j}$ which is opposite to $\ell_{ij}$. 
\end{theorem}
\begin{proof}
	Without loss of generality we assume $i=1, j=2$ and show that 
	\begin{equation*}
	\cos\theta_{34}=\frac{-c_{12}}{\sqrt{c_{11}\cdot c_{22}}}.
	\end{equation*}  
	If this is true then by permutation of vertices of $T$ the equation (iii) holds for the remining values of $i$ and $j$.
	
	By straightforward computation we get the following cofactors of matrix $E$
	\begin{equation}\begin{split}\label{eq31}
	-c_{12}=&\;(\ch\ell_{14}\ch\ell_{24}-\ch\ell_{12})\sh^2\ell_{34}\\
	&-(\ch\ell_{14}\ch\ell_{34}-\ch\ell_{13})(\ch\ell_{24}\ch\ell_{34}-\ch\ell_{23}),\\
	c_{11}=&\;\sh^2\ell_{24}\,\sh^2\ell_{34}-(\ch\ell_{24}\ch\ell_{34}-\ch\ell_{23})^2,\\
	c_{22}=&\;\sh^2\ell_{14}\,\sh^2\ell_{34}-(\ch\ell_{14}\ch\ell_{34}-\ch\ell_{13})^2.
	\end{split}\end{equation}
	
	Consider an intersection of the tetrahedron $T$ with a sufficiently small sphere centred at its vertex $4$ (see Fig.~3). 
	\begin{figure}[th]
		\centerline{\psfig{file=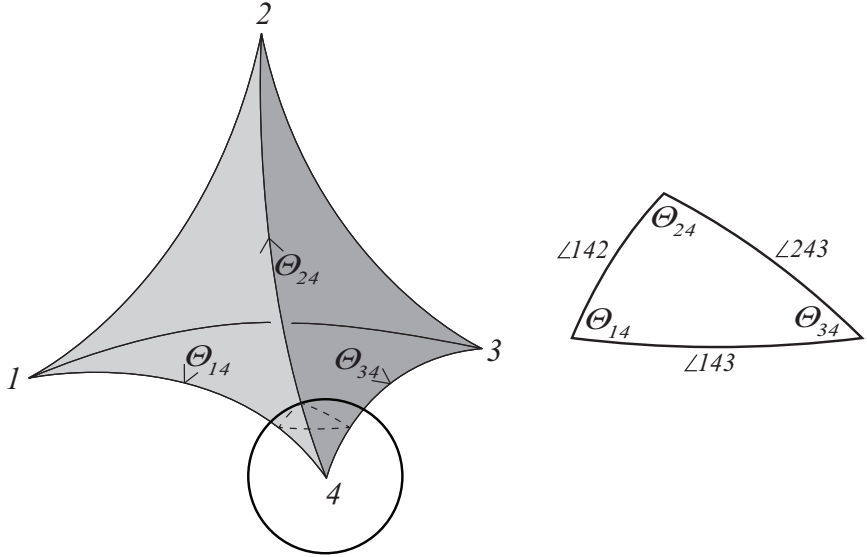,width=4.5in}}
		\caption{Section of a hyperbolic tetrahedron $T$ by a sphere centered at its vertex.\label{fig1}}
	\end{figure}
	The intersection bounds a spherical triangle whose angles are equal to dihedral angles along the edges of $T$, adjacent to its vertex $4$, namely $\theta_{14}, \theta_{24}$ and $\theta_{34}$. The sides of this spherical triangle have angular measures equal to the angles between corresponding edges of $T$ adjacent to vertex $4$, namely $\angle 243, \angle 143$ and $\angle 142$. By spherical law of cosines for this triangle we have
	\begin{equation}\label{F1}
	\cos\theta_{34}=\frac{\cos\angle 142-\cos\angle 143\,\cos\angle 243}{\sin\angle 143\,\sin\angle 243}.
	\end{equation}
	
	By hyperbolic law of cosines for the triangular faces $1-4-2, 1-4-3, 2-4-3$ of $T$ we have
	\begin{align}
	\cos\angle 142&=\frac{\ch\ell_{14}\,\ch\ell_{24} - \ch\ell_{12}}{\sh\ell_{14}\,\sh\ell_{24}}\,, \nonumber\\
	\cos\angle 143&=\frac{\ch\ell_{14}\,\ch\ell_{34} - \ch\ell_{13}}{\sh\ell_{14}\,\sh\ell_{34}}\,, \label{F3}\\
	\cos\angle 243&=\frac{\ch\ell_{24}\,\ch\ell_{34} - \ch\ell_{23}}{\sh\ell_{24}\,\sh\ell_{34}}\,. \nonumber
	\end{align}
	
	Since the angles $\angle 143, \angle 243$ are less than $\pi$, using (\ref{F3}) and (\ref{eq31}) we obtain
	\begin{equation}\begin{split}\label{eq34}
	\sin\angle 143&=\sqrt{1-\cos^2\angle 143}=\sqrt{\frac{\sh^2\ell_{14}\,\sh^2\ell_{34}-(\ch\ell_{14}\,\ch\ell_{34} - \ch\ell_{13})^2}{\sh^2\ell_{14}\,\sh^2\ell_{34}}}=\\
	&\frac{\sqrt{c_{22}}}{\sh\ell_{14}\,\sh\ell_{34}},\\
	\sin\angle 243&=\sqrt{1-\cos^2\angle 243}=\sqrt{\frac{\sh^2\ell_{24}\,\sh^2\ell_{34}-(\ch\ell_{24}\,\ch\ell_{34} - \ch\ell_{23})^2}{\sh^2\ell_{24}\,\sh^2\ell_{34}}}=\\
	&\frac{\sqrt{c_{11}}}{\sh\ell_{24}\,\sh\ell_{34}}.
	\end{split}\end{equation}
	
	Since $1-\cos^2\angle 243>0$ and $\sh^2\ell_{24}\,\sh^2\ell_{34}>0$, from the latter equations we conclude that $c_{11}>0$. Then by permutation of vertices of $T$ we conclude that inequality (i) $c_{ii}>0$ holds for any $i\in\{1,2,3,4\}$.
	
	Also, from (\ref{F3}) and (\ref{eq31}) we get
	\begin{equation}\begin{split}\label{eq35}
	&\cos\angle 142-\cos\angle 143\,\cos\angle 243=\\
	&\frac{(\ch\ell_{14}\,\ch\ell_{24} - \ch\ell_{12})\sh^2\ell_{34}-(\ch\ell_{14}\,\ch\ell_{34} - \ch\ell_{13})(\ch\ell_{24}\,\ch\ell_{34} - \ch\ell_{23})}{\sh\ell_{14}\,\sh\ell_{24}\,\sh^2\ell_{34}}=\\
	&\frac{-c_{12}}{\sh\ell_{14}\,\sh\ell_{24}\,\sh^2\ell_{34}}.
	\end{split}\end{equation}

	
	Substituting the expressions (\ref{eq34}) and (\ref{eq35}) into (\ref{F1}) we obtain 
	$
	\displaystyle\cos\theta_{34}=\frac{-c_{12}}{\sqrt{c_{11}\cdot c_{22}}},
	$
	which completes the proof of equation (iii).
	
	
	
	Now consider the edge matrix $E$. By Jacobi's equation (Theorem~\ref{Jac}, here we set parameter $k=2$) we have
	\begin{equation*}
	c_{22}c_{11}-c^2_{12}=\det E\,(1-\ch\ell^2_{34})=-\det E\,\sh^2\ell_{34}.
	\end{equation*}
	
	As we have shown, $\cos\theta_{34}=\dfrac{-c_{12}}{\sqrt{c_{11}\cdot c_{22}}}$. Let us rewrite this equation, passing to the new variable $\exp (i\theta_{34})$. We obtain a quadratic polynomial
	\begin{equation*}\label{F7}
	\exp^2(i\theta_{34}) +2 \frac{c_{12}}{\sqrt{c_{11}\cdot c_{22}}} \exp(i\theta_{34}) +1 =0.
	\end{equation*}
	By Vieta's formulas, the roots $z_1, z_2$ of this polynomial satisfy the relation $z_1\cdot z_2=1$. Therefore, the quadratic polynomial has two complex conjugate roots and the determinant is less than zero, namely
	\begin{equation}\label{F8}
	\frac{c^2_{21}-c_{11}\,c_{22}}{c_{11}\,c_{22}}<0.
	\end{equation}
	
	As we shown before, $c_{22}\,c_{11}>0$. Then the inequality (\ref{F8}) implies $c^2_{21}-c_{11}\,c_{22} = \det E\,\sh^2\ell_{34}<0$. Hence, $\det E<0$ and (ii) is proved. 
\end{proof}

The most useful formula (iii) form Theorem~\ref{edgematrix} can be also found in \cite{Mur-Ush}.

\section{Volume formula for a tetrahedron $T$ in $\mathbb{H}^3$}
The main result of the present work is the following. 
\begin{theorem}\label{vol}
	Let $T$ be a compact hyperbolic tetrahedron given by its edge matrix $E$ and $c_{ij}=(-1)^{i+j}E_{ij}$ is $ij$-cofactor of $E$. We assume that all the edge lengths are fixed exept $\ell_{34}$ which varies. Then the volume $V=V(T)$ is given by the formula 
	\begin{multline*}
	V=\frac12\int\limits_{\ell_1}^{\ell_{34}}\Bigl[\frac{-t}{\sqrt{-\Delta^3}}\Bigl(\frac{c_{14}(c_{11}c_{23}-c_{12}c_{13})}{c_{11}}+\frac{c_{24}(c_{13}c_{22}-c_{12}c_{23})}{c_{22}}\Bigr)-\frac{\sh t}{\sqrt{-\Delta}}\\
	\times\Bigl(\frac{\ell_{24}\sh\ell_{24}c_{14}+\ell_{14}\sh\ell_{23}c_{13}}{c_{11}}+\frac{\ell_{13}\sh\ell_{13}c_{23}+\ell_{23}\sh\ell_{14}c_{24}}{c_{22}}+\ell_{12}\sh\ell_{12}\Bigr)\Bigr]dt,
	\end{multline*}
	where cofactors $c_{ij}$ and edge matrix determinant $\Delta=\det E$ are functions in one variable $\ell_{34}$ denoted by $t$. The lower limit of integration $\ell_1$ is defined by expression
	\begin{multline*}
	\ch\ell_1=\ch \ell_{13}\ch \ell_{14}-\csch^2\ell_{12}\Bigl[(\ch \ell_{13}\ch \ell_{12}-\ch \ell_{23})(\ch \ell_{14}\ch \ell_{12}-\ch \ell_{24})\\
	-\sqrt{(\ch \ell_{23}-\ch(\ell_{13}+\ell_{12}))(\ch \ell_{23}-\ch(\ell_{13}-\ell_{12}))}\\
	\times\sqrt{(\ch \ell_{24}-\ch(\ell_{14}+\ell_{12}))(\ch \ell_{24}-\ch(\ell_{14}-\ell_{12}))}\Bigr].
	\end{multline*}
\end{theorem}
\begin{proof}
	A hyperbolic tetrahedron $T$ can be defined up to isometry by a set of dihedral angles (see, e.g., \cite{Vinberg}). According to Theorem \ref{edgematrix} the dihedral angles of $T$ are uniquely determined by its edge lengths
	\begin{equation}\label{F11}
	\cos\theta_{5-i,5-j}=\dfrac{-c_{ij}}{\sqrt{c_{ii}\cdot c_{jj}}}.
	\end{equation}
	We differentiate the volume of $T$ by $\ell_{34}$ as a composite function
	\begin{equation}\begin{split}\label{derivative}
	\frac{\partial V}{\partial \ell_{34}}=&\;\frac{\partial V}{\partial \theta_{14}}\frac{\partial \theta_{14}}{\partial \ell_{34}} + \frac{\partial V}{\partial \theta_{34}}\frac{\partial \theta_{34}}{\partial \ell_{34}} + \frac{\partial V}{\partial \theta_{24}}\frac{\partial \theta_{24}}{\partial \ell_{34}}\\
	&+\frac{\partial V}{\partial \theta_{13}}\frac{\partial \theta_{13}}{\partial \ell_{34}}+\frac{\partial V}{\partial \theta_{23}}\frac{\partial \theta_{23}}{\partial \ell_{34}}+\frac{\partial V}{\partial \theta_{12}}\frac{\partial \theta_{12}}{\partial \ell_{34}}.
	\end{split}\end{equation}
	
	By the Schl\"afli formula (see, e.g., \cite{Vinberg}, Ch.~7, Sect.~2.2), we have	
	\begin{equation*}\begin{split}\label{F9}
	d V=&-\sum_{ij}\frac{\ell_{ij}}{2}\,d\theta_{ij}=\\
	&-\frac{1}{2}(\ell_{14}\,d\theta_{14}\,+\,\ell_{34}\,d\theta_{34}\,+\,\ell_{24}\,d\theta_{24}\,+\,\ell_{13}\,d\theta_{13}\,+\,\ell_{23}\,d\theta_{23}\,+\,\ell_{12}\,d\theta_{12}),
	\end{split}\end{equation*} 
	where the sum is taken over all edges of $T$. Consequently, 
	\begin{equation}\label{F10}
	\dfrac{\partial V}{\partial \theta_{ij}}=-\frac{\ell_{ij}}{2}.
	\end{equation} 
	
	Applying Jacobi's equation (Theorem~\ref{Jac}) to matrix $E$ (we set parameter $k=2$) we obtain the following identities	
	\begin{equation}\begin{split}\label{Jac-rel}
	c_{11}c_{22}-c^2_{12}&=\Delta (1-\text{ch}^2 \ell_{34})=-\Delta \text{sh}^2 \ell_{34},\\
	c_{11}c_{33}-c^2_{13}&=\Delta (1-\text{ch}^2 \ell_{24})=-\Delta \text{sh}^2 \ell_{24},\\
	c_{22}c_{33}-c^2_{23}&=\Delta (1-\text{ch}^2 \ell_{14})=-\Delta \text{sh}^2 \ell_{14},\\
	c_{33}c_{44}-c^2_{34}&=\Delta (1-\text{ch}^2 \ell_{12})=-\Delta \text{sh}^2 \ell_{12},\\
	c_{22}c_{44}-c^2_{24}&=\Delta (1-\text{ch}^2 \ell_{13})=-\Delta \text{sh}^2 \ell_{13},\\
	c_{11}c_{44}-c^2_{14}&=\Delta (1-\text{ch}^2 \ell_{23})=-\Delta \text{sh}^2 \ell_{23},\\
	c_{14}c_{23}-c_{12}c_{34}&=\Delta (\ch \ell_{14}\ch \ell_{23}-\ch \ell_{12}\ch \ell_{34}),\\
	c_{13}c_{24}-c_{12}c_{34}&=\Delta (\ch \ell_{13}\ch \ell_{24}-\ch \ell_{12}\ch \ell_{34}),\\
	c_{13}c_{44}-c_{14}c_{34}&=-\Delta (\ch \ell_{13}-\ch \ell_{12}\ch \ell_{23}),\\
	c_{13}c_{14}-c_{11}c_{34}&=\Delta (\ch \ell_{34}-\ch \ell_{23}\ch \ell_{24}),\\
	c_{33}c_{14}-c_{34}c_{13}&=-\Delta (\ch \ell_{14}-\ch \ell_{12}\ch \ell_{24}),\\
	c_{23}c_{24}-c_{34}c_{22}&=\Delta (\ch \ell_{34}-\ch \ell_{13}\ch \ell_{14}),\\
	c_{23}c_{44}-c_{24}c_{34}&=-\Delta (\ch \ell_{23}-\ch \ell_{12}\ch \ell_{13}),\\
	c_{33}c_{24}-c_{34}c_{23}&=-\Delta (\ch \ell_{24}-\ch \ell_{12}\ch \ell_{14}).
	\end{split}\end{equation}
	
	Note that $\sin\theta_{5-i,5-j}\geq 0,\; i,j\in\{1,2,3,4\}$. Using (\ref{F11}) we find a partial derivative of a dihedral angle by variable
	$\ell_{34}$
	\begin{equation}\begin{split}\label{eq45}
	\frac{\partial\theta_{5-i,5-j}}{\partial\ell_{34}}=&-\frac{\sh\ell_{34}}{\sin\theta_{5-i,5-j}}\,\frac{\partial\cos\theta_{5-i,5-j}}{\partial\ch\ell_{34}}=-\frac{\sh\ell_{34}}{ \sqrt{1-\cos^2\theta_{5-i,5-j}} }\,\frac{\partial\cos\theta_{5-i,5-j}}{\partial\ch\ell_{34}}=\\
	&-\frac{\sh\ell_{34}\sqrt{c_{ii}\,c_{jj}}}{\sqrt{c_{ii}\,c_{jj}-c_{ij}^2}}\,\frac{\partial\cos\theta_{5-i,5-j}}{\partial\ch\ell_{34}}
	\end{split}\end{equation}
	
	Considering each cofactor of matrix $E$ as a function of variable $\ell_{34}$ we calculate the partial derivatives of these functions
	\begin{equation}\begin{split}\label{eq46}
	\frac{\partial c_{13}}{\partial\ell_{34}}&=\ch\ell_{14}-\ch\ell_{12}\ch\ell_{24},\quad\frac{\partial c_{11}}{\partial\ell_{34}}=2(\ch\ell_{23}\ch\ell_{24}-\ch\ell_{34}),\\
	\frac{\partial c_{14}}{\partial\ell_{34}}&=\ch\ell_{13}-\ch\ell_{12}\ch\ell_{23},\quad\frac{\partial c_{22}}{\partial\ell_{34}}=2(\ch\ell_{13}\ch\ell_{14}-\ch\ell_{34}),\\
	\frac{\partial c_{23}}{\partial\ell_{34}}&=\ch\ell_{24}-\ch\ell_{12}\ch\ell_{14},\quad\frac{\partial c_{33}}{\partial\ell_{34}}=0,\\
	\frac{\partial c_{24}}{\partial\ell_{34}}&=\ch\ell_{23}-\ch\ell_{12}\ch\ell_{13},\quad\frac{\partial c_{44}}{\partial\ell_{34}}=0,\\
	\frac{\partial c_{34}}{\partial\ell_{34}}&=\ch^2\ell_{14}-1=\sh^2\ell_{14},\quad\frac{\partial c_{12}}{\partial\ell_{34}}=2\ch\ell_{12}\ch\ell_{34}-\ch\ell_{13}\ch\ell_{24}-\ch\ell_{14}\ch\ell_{23}.
	\end{split}\end{equation}
	
	We substitute $\cos\theta_{5-i,5-j}$ in (\ref{eq45}) by formula (\ref{F11}), differentiate it and use expressions (\ref{eq46}) for cofactors derivatives. By straightforward computation we obtain 	
	\begin{equation*}\begin{split}
	\frac{\partial \theta_{14}}{\partial \ell_{34}}=&\,\frac{c_{44}(c_{11}(\ch\ell_{13}-\ch\ell_{12}\ch\ell_{23})+c_{14}(\ch\ell_{34}-\ch\ell_{23}\ch\ell_{24}))}{c_{44}c_{11}\sqrt{-\Delta}}\,\frac{\sh\ell_{34}}{\sh\ell_{14}},\\
	\frac{\partial \theta_{24}}{\partial \ell_{34}}=&\,\frac{c_{33}(c_{11}(\ch\ell_{14}-\ch\ell_{12}\ch\ell_{24})+c_{13}(\ch\ell_{34}-\ch\ell_{23}\ch\ell_{24}))}{c_{33}c_{11}\sqrt{-\Delta}}\,\frac{\sh\ell_{34}}{\sh\ell_{24}},\\
	\frac{\partial \theta_{34}}{\partial \ell_{34}}=&\,\frac{(2\ch\ell_{12}\ch\ell_{34}-\ch\ell_{14}\ch\ell_{23}-\ch\ell_{13}\ch\ell_{24})c_{22}c_{11} }{c_{11}c_{22}\sqrt{-\Delta}}\\
	&  -\frac{c_{12}(c_{11}(\ch\ell_{13}\ch\ell_{14}-\ch\ell_{34})+c_{22}(\ch\ell_{23}\ch\ell_{24}-\ch\ell_{34}))}{c_{11}c_{22}\sqrt{-\Delta}},\\
	\frac{\partial \theta_{13}}{\partial \ell_{34}}=&\,\frac{c_{44}(c_{22}(\ch\ell_{23}-\ch\ell_{12}\ch\ell_{13})+c_{24}(\ch\ell_{34}-\ch\ell_{13}\ch\ell_{14}))}{c_{44}c_{22}\sqrt{-\Delta}}\,\frac{\sh\ell_{34}}{\sh\ell_{13}},\\
	\frac{\partial \theta_{23}}{\partial \ell_{34}}=&\,\frac{c_{33}(c_{22}(\ch\ell_{24}-\ch\ell_{12}\ch\ell_{14})+c_{23}(\ch\ell_{34}-\ch\ell_{13}\ch\ell_{14}))}{c_{33}c_{22}\sqrt{-\Delta}}\,\frac{\sh\ell_{34}}{\sh\ell_{23}},\\
	\frac{\partial \theta_{12}}{\partial \ell_{34}}=&\,\frac{\ch\ell^2_{12}-1}{\sqrt{-\Delta}}\,\frac{\sh\ell_{34}}{\sh\ell_{12}}.
	\end{split}\end{equation*} 
	
	We rewrite these partial derivatives using relations (\ref{Jac-rel})
	
	\begin{equation}\begin{split}\label{PD}
	\frac{\partial \theta_{14}}{\partial \ell_{34}}&=\frac{c_{13}\sh \ell_{23}\sh\ell_{34} }{c_{11}\sqrt{-\Delta}},\\
	\frac{\partial \theta_{24}}{\partial \ell_{34}}&=\frac{c_{14}\sh \ell_{24}\sh\ell_{34} }{c_{11}\sqrt{-\Delta}},\\
	\frac{\partial \theta_{34}}{\partial \ell_{34}}&=\Bigl[\frac{c_{14}(c_{11}c_{23}-c_{12}c_{13})}{c_{11}}+\frac{c_{24}(c_{13}c_{22}-c_{12}c_{23})}{c_{22}}\Bigr]\frac{1}{\sqrt{-\Delta^3}},\\
	\frac{\partial \theta_{13}}{\partial \ell_{34}}&=\frac{c_{23}\sh \ell_{13}\sh\ell_{34} }{c_{22}\sqrt{-\Delta}},\\
	\frac{\partial \theta_{23}}{\partial \ell_{34}}&=\frac{c_{24}\sh \ell_{14}\sh\ell_{34} }{c_{22}\sqrt{-\Delta}},\\
	\frac{\partial \theta_{12}}{\partial \ell_{34}}&=\frac{\sh \ell_{12}\sh \ell_{34} }{\sqrt{-\Delta}}.
	\end{split}\end{equation} 
	
	Substituting the partial derivatives (\ref{F10}) and (\ref{PD}) into (\ref{derivative}) we obtain
	
	\begin{multline*}
	\frac{\partial V}{\partial \ell_{34}}= \frac{1}{2}\Bigl[ \frac{-\ell_{34}}{\sqrt{-\Delta^3}}\Bigl(\frac{c_{14}(c_{11}c_{23}-c_{12}c_{13})}{c_{11}}+\frac{c_{24}(c_{13}c_{22}-c_{12}c_{23})}{c_{22}}\Bigl)-\frac{\sh\ell_{34}}{\sqrt{-\Delta}}\\
	\times\Bigl(\frac{\ell_{24}\sh\ell_{24}c_{14}+\ell_{14}\sh\ell_{23}c_{13}}{c_{11}}+\frac{\ell_{13}\sh\ell_{13}c_{23}+\ell_{23}\sh\ell_{14}c_{24}}{c_{22}}+\ell_{12}\sh\ell_{12}\Bigr) \Bigr].
	\end{multline*}
	
	Now we assume that all the edge lengths of $T$ are fixed except $\ell_{34}$ which varies. According to the Theorem~\ref{exist} we have $\ell_{34} \in [\ell_1,\ell_2]$.  Note that the volume $V=0$ when $\ell_{34}=\ell_1$. So we can obtain the volume of $T$ by integrating the function $\displaystyle\frac{\partial V}{\partial \ell_{34}}$ of variable $\ell_{34}$ from $\ell_1$ to a necessary value of $\ell_{34}$. For distinguishing between variable of integration with its limits we denote the variable $\ell_{34}$ by $t$. Thus we obtain the volume formula for the hyperbolic tetrahedron $T$ as in the statement of this theorem.
\end{proof}

Now let us put every edge length to be equal $\ell_{ij}=a$. Then form Theorem~\ref{vol} we have the following.

\begin{corollary}
	Let $T$ be a regular hyperbolic tetrahedron and all of its edge lengths are equal to $a,\; a\geq0$. Then the volume $V=V(T)$ is given by the formula
	\begin{align*}
	V&=\frac{1}{2}\int\limits_{0}^{a} \frac{(A-B)dt}{C\sqrt{D}},\quad where\\
	A&=2t\ch^2 a \sqrt{(\ch a-1)(\ch t -1)},\\
	B&=a(1-4 \ch a 2+ \ch^2 a+\ch t) \sqrt{(\ch a+1)(\ch t +1)},\\
	C&=(1- \ch^2 a + \ch t),\\
	D&=4\ch^2 a-\ch a -1 -\ch t - \ch a \ch t.
	\end{align*}
\end{corollary}
This coincides with the result of \cite{AbrVuo2017} (Theorem~1).

\appendix

\section*{Acknowledgments}
The authors are grateful to Alexander Mednykh for useful remarks and comments. This work was supported by the Ministry of Science and Higher Education of Russia (agreement No. ~075-02-2020-1479/1).

\end{document}